\documentclass[11pt]{amsart}

\usepackage{amsmath}	
\usepackage{amssymb}	
\usepackage{amsthm}
\usepackage[utf8]{inputenc}
\usepackage[dvips]{graphicx}
\usepackage{pgfplots}
\usepackage{url}
\usepackage{mathtools}
\usepackage[english]{babel}
\usepackage{multirow,tabu}
\usepackage{float}
\usepackage{placeins}
\usepackage{amsfonts}
\usepackage{verbatim}
\usepackage[T1]{fontenc} 	
\usepackage{csquotes}
\usepackage{amscd}
\usepackage{enumerate}
\usepackage{enumitem}
\usepackage{calc}

\newtheorem{thm}{Theorem}[section]
\newtheorem{cor}[thm]{Corollary}
\newtheorem{lem}[thm]{Lemma}

\theoremstyle{definition}

\newtheorem{rem}[thm]{Remark}

\numberwithin{equation}{section}

\newcommand{\norm}[1]{\left\Vert#1\right\Vert}
\newcommand{\abs}[1]{\left\vert#1\right\vert}

\newcommand{\B}{\mathcal{B}}

\newcommand*\HH{\mathcal{H}} 

\newcommand{\field}[1]{\mathbb{#1}}

\newcommand{\C}{\field{C}}

\newcommand{\Z}{\field{Z}}
\newcommand{\D}{\field{D}}

\newcommand*\conj[1]{\overline{#1}}
\newcommand*\closed[1]{\overline{#1}}

\let\latexchi\chi
\makeatletter
\renewcommand\chi{\@ifnextchar_\sub@chi\latexchi}
\newcommand{\sub@chi}[2]{
  \@ifnextchar^{\subsup@chi{#2}}{\latexchi^{}_{#2}}%
}
\newcommand{\subsup@chi}[3]{
  \latexchi_{#1}^{#3}%
}
\makeatother

\DeclareMathOperator{\id}{id}
\DeclareMathOperator{\spa}{span}

\begin{document}

\title[Unified approach to spectral properties of multipliers]{Unified approach to spectral properties  of multipliers}

\author[Lindstr{\"o}m] {Mikael Lindstr\"om}\address{Mikael
Lindstr{\"o}m. Department of Mathematics,  \AA bo Akademi
University. FI-20500 \AA  bo, Finland. \emph{e}.mail:
mikael.lindstrom@abo.fi}
\author[Miihkinen]{Santeri Miihkinen}\address{Santeri Miihkinen.
Department of Mathematics,  Abo Akademi University. FI-20500 \AA  bo,
Finland. \emph{e}.mail: santeri.miihkinen@abo.fi}
\author[Norrbo] {David Norrbo}
\address{David Norrbo. Department of Mathematics, \AA bo Akademi University. FI-20500 \AA bo, Finland. \emph{e}.mail:  	dnorrbo@abo.fi }






\begin{abstract} 
Let  $\mathbb B_n$ be the open unit  ball  in $\mathbb C^n$. We characterize the spectra of pointwise multipliers $M_u$ acting on Banach spaces of analytic functions on $\mathbb B_n$ satisfying some general conditions. These spaces include Bergman-Sobolev  spaces $A^p_{\alpha,\beta}$, Bloch-type spaces $\mathcal B_\alpha$, weighted Hardy spaces $H^p_w$ with Muckenhoupt weights and Hardy-Sobolev Hilbert spaces $H^2_\beta$.  Moreover, we describe the essential spectra of multipliers in most of the aforementioned spaces, in particular, in those spaces for which the set of multipliers is a subset of the ball algebra.

\end{abstract}

\maketitle

\section{Introduction and preliminaries}

 In a very recent article  \cite{CHZ}, Cao, He, and Zhu considered the multiplication operator $M_u$ acting on the Hardy-Sobolev Hilbert space and characterized the spectrum and essential spectrum of $M_u.$ In the present work, we extend and generalize the results obtained there from Hardy-Sobolev Hilbert space to the Bergman-Sobolev and Bloch-type spaces of the open unit ball $\mathbb{B}_n$ of $\C^n$ and weighted Hardy spaces of the open unit disk $\D$ with Muckenhoupt weights. In particular, our main focus is to allow the multiplier space $M(X(\mathbb B_n))$ to be contained in the ball algebra, which holds for example for certain Bergman-Sobolev spaces and Bloch-type spaces. We formulate our results on spectral properties of $M_u$ acting on a Banach space $X(\mathbb{B}_n)$ of analytic functions in $\mathbb{B}_n$, where $X(\mathbb{B}_n)$ satisfies very general and natural properties regarding its multiplier space and the norm topology. Consequently, we approach the spectral properties of multipliers in a unified manner and  
key examples of such spaces include the aforementioned spaces. Aside from obtaining a description of the spectrum for all spaces satisfying the mentioned properties, we also have to develop some new techniques to determine the essential spectrum of $M_u$ regarding the non-Hilbert space case. Other previous work regarding spectral and related properties of multiplication operators on analytic function spaces includes \cite{A}, \cite{BB1}, \cite{BDL}, \cite{CH}, \cite{FX}, \cite{OF} and \cite{V}.

The article is organised as follows. In section 2, we introduce general Banach spaces $X(\mathbb B_n)$ of analytic functions on $\mathbb B_n$ and give central concrete examples of them. Section 3 focuses on the  spectrum of $M_u$ by first establishing a characterization of invertibility of $M_u$ and then obtaining  the spectrum of $M_u$ and giving admissible examples of spaces on which $M_u$ can be defined. In section 4, we begin with a characterization of the essential spectrum in the high-dimensional case $n > 1$. Then we consider the case $n=1$ by first establishing a characterization of the Fredholmness of $M_u$ when conditions \textbf{(I)}, \textbf{(IV)} and $M(X(\D))= H^\infty(\D)$ hold. Examples of spaces satisfying the previous conditions are also given. Next, we consider the difficult case when $M(X(\D)) \subset A(\D)$ (or $u \in M(X(\D)) \cap A(\D)$) and starting off with the space $X(\D)=\mathcal{B}_\alpha (\D)$ for $0 < \alpha \le 1$ and showing that the condition, earlier observed to be sufficient for the Fredholmness of $M_u$, is also necessary. Finally, we show the necessity of the condition in the case of  those Bergman-Sobolev spaces $A^p_{\alpha,\beta}(\D)$ for which $M(A^p_{\alpha,\beta}(\D)) \subset A(\D).$ From these two cases we obtain the essential spectrum of $M_u$ for several scales of spaces $\mathcal{B}_\alpha(\D)$ and $A^p_{\alpha,\beta}(\D)$ as the main result of Section 4.

To conclude, our main result regarding the spectra of multiplication operators acting on $X(\mathbb B_n)$ is Theorem \ref{MainThmSpec}. The essential spectra of operators $M_u$ acting on certain spaces $X(\D)$ having their multiplier spaces $M(X(\D))$ contained in the disk algebra are described in Theorem \ref{MainResEssSpec}. In the case of general $X(\D)$ with $M(X(\D)) = H^\infty(\D)$, the essential spectra of operators $M_u$ are characterized in Theorem \ref{EssSpecMlutiHInfty}. In Theorem \ref{EssSpecForMultiDim}, we present the high-dimensional case $n>1$ concerning the essential spectra of operators $M_u$ acting on general spaces $X(\mathbb B_n)$.

Now we introduce some definitions and notations. Throughout this article, let $\mathbb Z_{\geq a} = \{n\in\mathbb Z : n\geq a\}$ and $\mathbb Z_{>a}=\{n\in\mathbb Z : n> a\}$, where $a\in\mathbb R$. Furthermore, let  $\mathbb B_n = \{z\in\mathbb{C}^n : \abs{z}<1 \}$, $n\in \mathbb Z_{\geq 1}$, be the open unit  ball  in $\mathbb C^n$ and $\D = \mathbb B_1$. Moreover,  let $\HH(\mathbb B_n)$ be the space  of all analytic  functions $f\colon \mathbb B_n\to \mathbb{C}$ and ${\mathcal P}(\mathbb B_n)$ be the set of all analytic polynomials  $p\colon\mathbb B_n\to \mathbb C$ such that $p(z)=\sum_{k \in J} c_k z^k$, where $J \subset \Z^n_{\geq 0}$ is a finite set, $k = (k_1,\ldots, k_n) \in \Z^n_{\geq 0}, \, |k| = k_1+\ldots+k_n$, $z^k = z_1^{k_1}\cdot \ldots \cdot z_n^{k_n}$ and  $c_k\in \mathbb C$ for  $ k \in J$. 

We also recall that a bounded linear operator $T$ acting on a Banach space is Fredholm if it has  closed range and both kernel and cokernel of $T$ are finite dimensional. The essential spectrum $\sigma_e(T)$ of an operator $T$ is defined as $\sigma_e(T)=\{\lambda \in \C: T-\lambda I \textup{ is not Fredholm}\},$ where $I$ is the identity operator, and the reader may observe that $\sigma_e(T)$ is a subset of the spectrum $\sigma(T)$. See \cite{AA} for more details on Fredholm properties of bounded operators.

For any  $f\in \HH (\mathbb B_n),$ the gradient of $f$  is given by  
\[
\nabla f(z) = \left( \frac{\partial f}{\partial {z_1}},...,  \frac{\partial f}{\partial {z_n}}\right)
\]
and will be denoted $Df(z)$ in the case $n=1$.

Let $\beta\in\mathbb{R}$ and $f\in \HH(\mathbb B_n)$. The fractional radial derivative $R^\beta$ is given by
\[
R^\beta f (z) = \sum_{k=1}^\infty k^\beta f_k(z),
\]
where  $f(z)=\sum_{k=0}^\infty f_k(z)$  is the homogeneous expansion of $f\in \HH(\mathbb B_n)$. Let $I\colon \HH(\mathbb B_n) \to  \HH(\mathbb B_n) $ be the identity operator. The operator $(I+R)^\beta$ will also be used and is naturally defined by
\[
(I+R)^\beta f (z) = \sum_{k=0}^\infty (1+k)^\beta f_k(z).
\]

\smallskip

For expressing asymptotic behaviour, the notation $a_k\sim b_k$, as $k\to\infty$, means $\lim_{k\to\infty} \frac{a_k}{b_k}=1$. Moreover, by $a(x)\gtrsim b(x)$ (or $a(x)\lesssim b(x)$) we indicate the existence of a constant $C>0$ independent of $x$ such that $a(x)\geq C b(x) $ (or $a(x)\leq C b(x)$) for all $x$ in some implicit set. If both $a(x)\gtrsim b(x)$ and $a(x)\lesssim b(x)$ hold, we write $a(x)\asymp b(x)$. When two Banach spaces $X_1$ and $X_2$ are isomorphic, we use the notation $X_1 \simeq  X_2$.

\bigskip

\section{Conditions and Examples}

We deal with a vector space $X(\mathbb B_n)$ of analytic functions on $\mathbb B_n$ and a norm $\|\cdot\|_X$ on it, that renders $X(\mathbb B_n)$ a Banach space. 
As usual, for each $z\in \mathbb B_n,$ the evaluation functional $\delta_z$ is defined by $\delta_z(f) = f(z)$ for all $f \in X(\mathbb B_n)$. We assume that  $X(\mathbb B_n)$ contains the constant functions, so then all $\delta_z$ are non-zero.  Furthermore, we associate to $X(\mathbb B_n)$ another Banach space $Y(\mathbb B_n)\subset \HH(\mathbb B_n)$  containing the constant functions and equipped with the norm $\norm{\cdot}_Y$ as will be explained  below. 

\bigskip

The Banach spaces $X(\mathbb B_n)$  and  $Y(\mathbb B_n)$  are often assumed  to satisfy the first three conditions below:

\begin{enumerate}[label=(\textbf{\Roman*})]

\item\label{finerThanCOTop} \emph{The topologies induced by $\norm{\cdot}_X$ and $\norm{\cdot}_Y$ are both finer than the compact-open topology $\tau_0.$  }
\end{enumerate}

\smallskip
\noindent In particular,  for every $z\in \mathbb B_n$,
$\delta_z$ is a bounded linear  functional on both  $X(\mathbb B_n)$ and  $Y(\mathbb B_n)$.

\smallskip

 Let 
\[
M(X(\mathbb B_n))=\left\{    u\in \HH(\mathbb B_n) :  uf \in X(\mathbb B_n)   \  \text{for all}     \  f\in X(\mathbb B_n)  \right\}.
\]
Using condition \ref{finerThanCOTop} and the closed graph theorem, it follows that  every $u\in M(X(\mathbb B_n))$ induces a bounded linear operator $M_u: X(\mathbb B_n)\to X(\mathbb B_n).$

\smallskip

\begin{enumerate}[resume, label=(\textbf{\Roman*})]
\item \label{RelationXtoY} \emph{For  some $N\in \mathbb Z_{\geq 1}$ it holds that $\norm{f}_X \asymp \abs{f(0)}+\norm{R^N f}_Y$ for all $f\in \HH(\mathbb B_n)$.}
\end{enumerate}

Condition \ref{RelationXtoY} describes a relationship between the Banach spaces $X(\mathbb B_n)$ and $Y(\mathbb B_n)$ such that Lemma \ref{MuInvert} holds. Since the lemma is trivial for spaces $X(\mathbb B_n)$ with $ M(X(\mathbb B_n))=H^\infty (\mathbb B_n)$, this condition may be omitted when such spaces are considered. For these spaces we have $Y(\mathbb B_n)=X(\mathbb B_n)$.

\smallskip
\begin{enumerate}[resume, label=(\textbf{\Roman*})]
\item \label{HinftySubsetMultiplY} \emph{$H^\infty(\mathbb B_n)\subset M(Y(\mathbb B_n))$.}
\end{enumerate}

\smallskip

By condition \ref{finerThanCOTop} it is well-known that $\sup_{z\in\mathbb B_n}|u(z)|\le ||M_u||$ for all $u\in M(X(\mathbb B_n))$, so   
$M(X(\mathbb B_n))\subset H^\infty(\mathbb B_n) \ \ \text{ and } \ \ M(Y(\mathbb B_n)) = H^\infty(\mathbb B_n) ,$ where also  condition \ref{HinftySubsetMultiplY} is used in the second
statement. Since $u\mapsto M_u $ is bounded according to the bounded inverse theorem, it follows from the boundedness of $M_u$, that there exists a constant $C>0$ such that $\|M_ug\|_Y\le  C\|u\|_\infty  ||g||_Y$ for all $g\in Y(\mathbb B_n)$ and $u\in M(Y(\mathbb B_n))$.

\smallskip

When considering the case $n=1$, we will need the following condition to determine the essential spectra of the multiplication operator generated by $u\in M(X(\D))$.
\smallskip
\begin{enumerate}[resume, label=(\textbf{\Roman*})]
\item \label{divideByPolynom} \emph{If $f\in X(\mathbb D)$ has a zero  at $z_0\in \mathbb D$, then $\frac{f(z)}{z- z_0}\in X(\mathbb D).$}
\end{enumerate}
\smallskip

\begin{lem}\label{dividingWithPolynomial}
Let $f\in \mathcal H (\mathbb D )$ and $v \colon \mathbb D \to [0,\infty)$ be a bounded function such that $v(z)=v(|z|)$ for all $z\in\mathbb D$. Moreover, let $N\in\mathbb Z_{\geq 0}$ be such that 
\[
\sup_{z\in\mathbb{D}} v(z) |D^N f(z)| <\infty.
\]
If $z_0\in \mathbb D$ is a zero of $f$, then
\[
\sup_{z\in\mathbb{D}}v(z) \left|D^N\frac{f(z)}{z-z_0}\right|<\infty.
\]

\end{lem}
\begin{proof}
Let $g(z)=\frac{f(z)}{z-z_0}$. Since $f$ is analytic with a zero at $z_0$ we have $D^N g \in \mathcal H (\mathbb D)$. Thus, $h(z) = v(z) |D^N g(z)|$ is bounded on $\mathbb D$ if and only if $h$ is bounded near the boundary. For $z\in T=\left\{ z\in\mathbb D:|z|>\frac{1+|z_0|}{2}\right\}$, we have the following estimate

\begin{equation}
\label{GBlochUpperApprox}
|D^N g (z)| = \left|\sum_{j=0}^N \binom{N}{j} D^j f (z) D^{N-j}(z-z_0)^{-1}\right| \leq \sum_{j=0}^N \frac{ N!|D^j f(z)| }{  (|z|-|z_0|)^{N-j+1}   }.
\end{equation}
Furthermore, for $k\geq 0$ we have

\begin{align*}
|D^{k} f (z)| &\leq \left| \int_{C_z} D^{k+1} f(w) dw\right| + |D^k f(0)| \\
&\leq |z| \sup_{|w|=|z|}| D^{k+1} f(w)| + |D^k f(0) |\\
&\leq \sup_{|w|=|z|}| D^{k+1} f(w)| + |D^k f(0) |,
\end{align*}
where $C_z$ is the line from $0$ to $z$ in $\mathbb D$. 

By induction, it can be shown that
\[
|D^{k} f (z)| \leq \sup_{|y|=|z|}  | D^{N} f(y)| + \sum_{j=0}^{N-k-1}|D^{k+j}f(0) | 
\]
for $0\leq k \leq N$. Moreover, from the fact that $\sup_{z\in\mathbb D}\sup_{|w|=|z|}$ is interchangable with $\sup_{w\in\mathbb D}$ and $v(z) = v(|z|)$ we now obtain

\begin{align*}
 v(z) |D^{k} f (z)| &\leq \sup_{|y|=|z|} v(z) | D^{N} f(y)| + v(z)\sum_{j=0}^{N-k-1}|D^{k+j}f(0) |\\
&= \sup_{|y|=|z|} v(y) | D^{N} f(y)| + v(z)\sum_{j=0}^{N-k-1}|D^{k+j}f(0) |\\
&\leq \sup_{z\in \mathbb D}  v(z)| D^{N} f(z)| + \sup_{z\in \mathbb D}  v(z) \sum_{j=0}^{N-1}|D^{j}f(0) |= M_{f,N,v}<\infty
\end{align*}
for all $z\in\mathbb D$. Especially for $z\in T$, using (\ref{GBlochUpperApprox}), we have

\begin{align*}
 v(z)|D^N g (z)|&\leq \sum_{k=0}^N \frac{N!  v(z) |D^k f(z)| }{   (|z|-|z_0|)^{N-k+1}   }\leq  M_{f,N,v}\sum_{k=0}^N \frac{N!2^{N-k+1}   }{   (1-|z_0|)^{N-k+1}   }<\infty,
\end{align*}
which proves the lemma.

\end{proof}

Next, we list a number of spaces satisfying the above conditions \ref{finerThanCOTop}-\ref{divideByPolynom}. However, in part \textbf{(c)} we consider 
spaces $X(\mathbb B_n)$ for which $M(X(\mathbb B_n))=H^\infty(\mathbb B_n)$, implying that condition \ref{RelationXtoY} is irrelevant. 

\smallskip

{\bf  Examples.} $\textbf{(a)}$ For  $\alpha>0$  the Bloch-type space $X(\mathbb B_n)=\mathcal B _\alpha(\mathbb B_n)$ is the space of all  $f\in \HH (\mathbb B_n)$ satisfying
$\norm{f}_{\mathcal B_ \alpha}=|f(0)| + \sup_{z\in\mathbb B_n}(1-\abs{z}^2)^\alpha \abs{\nabla f(z)} < \infty$, see \cite{Z}. To these spaces correspond
\[
Y(\mathbb B_n)= H_{\alpha}^\infty(\mathbb B_n)=\{f\in\mathcal H(\mathbb B_n) : \norm{f}_{H^\infty_\alpha}=\sup_{z\in\mathbb B_n} (1-\abs{z}^2)^\alpha \abs{f(z)}<\infty\},
\]
see \cite{BS}. The little  Bloch-type space 
$\mathcal B _{0,\alpha}(\mathbb B_n)$  is the subspace of $\mathcal B _\alpha(\mathbb B_n)$ satisfying
$\lim_{|z|\to1} (1-\abs{z}^2)^\alpha \abs{\nabla f(z)}=0$. It is well-known that these spaces obey \ref{finerThanCOTop}. Let $\norm{f}_{\B R_\alpha}=|f(0)| + \sup_{z\in\mathbb B_n}(1-\abs{z}^2)^\alpha \abs{Rf(z)}$. According to Theorem 7.1 in \cite{Z}, it holds that 
\[
\{f\in \mathcal H (\mathbb B_n):\norm{f}_{\B_\alpha}<\infty\}=\{f\in \mathcal H (\mathbb B_n):\norm{f}_{\B R_\alpha}<\infty\}.
\]
Therefore it follows from the bounded inverse theorem that $\norm{\cdot}_{\mathcal B _\alpha}\asymp \norm{\cdot}_{\B R_\alpha}$ and, hence, both of these spaces satisfy condition \ref{RelationXtoY}. Condition \ref{HinftySubsetMultiplY} holds by definition. We will consider the space $(\B_\alpha(\mathbb D),\norm{\cdot}_{\mathcal B _\alpha})$ in the one-dimensional case.

By  Theorem~{2.1~(i)} in \cite{OSZ}, for $0 < \alpha < 1$, 
$u \in M(\B_\alpha(\D))  \text{  if and only if} \  u \in \B_\alpha(\D) \cap H^\infty(\D)=  \B_\alpha(\D)\subset A(\D)$, where the inclusion is found in  Theorem 7.9 in \cite{Z}.
When $\alpha = 1$, we get from     Theorem~{2.1~(ii)} in \cite{OSZ} that  $u \in M(\B_\alpha(\D))$   if and only if 
\[
\sup_{z\in\D} \abs{u'(z)}(1 - \abs{z}^2)\log\left(\frac{e}{1 - \abs{z}^2}\right) < \infty 
	\quad \text{and} \quad  u\in  H^\infty(\D).
\] 
Therefore $u\in \B_{0,1}(\D).$ Finally by    Theorem~{2.1~(iii)} in \cite{OSZ} we have for $\alpha > 1$ that  
$u \in M(\B_\alpha(\D))  \text{ if and only if}  \   u \in \B_1(\D) \cap H^\infty(\D)=  H^\infty(\D).$ According to Lemma \ref{dividingWithPolynomial} a function belonging to $\B_\alpha(\D)$ will remain in $\B_\alpha(\D)$ after removing a finite number of zeros $z_0$ through division by $z-z_0$, which proves \ref{divideByPolynom}. The notations $\B(\D)$ and $\B_0(\D)$ stand for $\B_1(\D)$ and $\B_{0,1}(\D)$ respectively.
\\

\smallskip

$\textbf{(b)}$ Let $\beta\ge 0$, $\alpha\ge -1$  and $1\leq p <\infty$. The holomorphic Sobolev space $A_{\alpha,\beta}^p(\mathbb B_n)$  is defined by
\[
A_{\alpha,\beta}^p(\mathbb B_n) = \{  f\in\HH(\mathbb B_n) : \norm{f}_{A_{\alpha,\beta}^p}  <\infty \},
\]
where the norm is defined by
\[
\norm{f}_{A_{\alpha,\beta}^p} =\norm{(I+R)^\beta f}_{A_\alpha^p} =\left( \int_{\mathbb B_n} |(I+R)^\beta f(z)|^p  dA_\alpha(z)  \right)^\frac{1}{p}
\]
for $\alpha>-1$ and
\[
\norm{f}_{A_{-1,\beta}^p} =\norm{(I+R)^\beta f}_{H^p} =\left( \int_{\partial\mathbb B_n} |(I+R)^\beta f(z)|^p  dS(z)  \right)^\frac{1}{p}.
\]
Furthermore, $dA_\alpha(z)=  \frac{\Gamma(n+\alpha+1)}{n! \Gamma(\alpha+1)} (1-\abs{z}^2)^\alpha dA(z)$, where $dA(z)$ is the $2n$-dimensional Lebesgue measure normalized so that $\int_{\mathbb B_n} dA(z) = 1$, and hence, $\int_{\mathbb B_n} dA_\alpha(z) = 1$ for every $\alpha>-1$. The notation $dS(z)$ stands for the surface measure satisfying $\int_{\partial \mathbb B _n} dS(z)=1$. 
The holomorphic Sobolev spaces can be partitioned into the Bergman-Sobolev spaces, $\alpha>-1$, and the Hardy-Sobolev spaces, $H^p_\beta(\mathbb B_n) =A^p_{-1,\beta}(\mathbb B_n) .$ In case of $\beta=0$, these spaces are called the weighted Bergman spaces $A^p_\alpha(\mathbb B_n) = A^p_{\alpha,0}(\mathbb B_n)$ with $\alpha > -1$ and the Hardy spaces $H^p (\mathbb B_n) = A^p_{-1,0}(\mathbb B_n)$.

For $p\ge 1$, $\alpha_j > -1$, $\beta_j \ge 0$ $(j=1,2)$, with $\alpha_1- \alpha_2 = p(\beta_1 -\beta_2)$, the following equivalence holds by Theorem 5.12 in \cite{BB} (see also \cite{CKS}):
\begin{equation} \label{formimp}
A_{\alpha_1,\beta_1}^p(\mathbb B_n)   \simeq A_{\alpha_2,\beta_2}^p(\mathbb B_n),
\end{equation}
where the isomorphism is given by the identity operator, and hence, the spaces have equivalent norms. By the same theorem, one also obtains the statement (\ref{formimp}) for $\alpha_1=-1$ and $p=2$. From this it follows that for $\beta_1 < \frac{1+\alpha_1}{p}$, where equality may be used in the case of $p=2$, we have $A_{\alpha_1,\beta_1}^p(\mathbb B_n)   \simeq A_{\alpha_1-\beta_1 p,0}^p(\mathbb B_n)$. The right-hand side is a weighted Bergman space or $H^2(\mathbb B_n)$, hence, $M(A_{\alpha,\beta}^p(\mathbb B_n)) = H^\infty (\mathbb{B}_n)$ for $\beta<\frac{1+\alpha}{p}$, where equality may be used in the case of $p=2$. Regarding the case $n=1$, if $\beta > \frac{2+\alpha}{p}$, then $A_{\alpha,\beta}^p(\D)$ is an algebra and $M(A_{\alpha,\beta}^p(\D))= A_{\alpha,\beta}^p(\D)$, see \cite{BB1}. In this setting, there is a $b<\beta$ satisfying $0 < b-\frac{2+\alpha}{p} <1$, so that 
\[
A_{\alpha,\beta}^p(\D) \subset A_{\alpha,b}^p(\D) \subset \Lambda_{ b-\frac{2+\alpha}{p}}(\D) \subset A(\D),
\] 
where $\Lambda_{ b-\frac{2+\alpha}{p}}$ is a Lipschitz space, see \cite{Z}. The first inclusion follows from (\ref{formimp}), the second inclusion can be found in Theorem 5.5 in \cite{BB} and the last one is given by Theorem 7.9 in \cite{Z}. Furthermore,  by Proposition 2.2 in \cite{CKS}, we have for $p\ge 1$, $\alpha\ge -1,$ and every positive integer $N$ that
\begin{equation}\label{good}
\norm{f}_{A_{\alpha,N}^p} \asymp \sum_{j=0}^{N-1} |D^jf(0)| + \norm{D^Nf}_{A_\alpha ^p}
\end{equation}
for $f\in\HH(\mathbb D)$. Next, we check the conditions \ref{finerThanCOTop} - \ref{divideByPolynom}.

The topology generated by $\norm{\cdot}_{A^p_{\alpha,\beta}}$ is finer than the compact-open topology $\tau_0$, so condition \ref{finerThanCOTop} holds. Indeed, the statement follows from Lemma 5.6 in \cite{BB} with the use of supremum over an arbitrary compact subset of $\mathbb B_n$. Hereafter, we will assume that $\beta\geq \frac{1+\alpha}{p}$. For smaller $\beta$ it was mentioned that the multiplier space is $H^\infty(\mathbb B_n)$ which is considered in \textbf{Examples} \textbf{(c)}, where the the space $A_{\alpha,\beta}^p(\mathbb B_n)$ can be viewed as a weighted Bergman space.

 In the case $N>\beta -\frac{\alpha+1}{p}\geq 0$, an application of  (\ref{formimp})  gives that 
\[
f\in A_{\alpha,\beta}^p(\mathbb B_n)  \ \text{if and only if} \   f\in A_{(N-\beta)p+\alpha,N}^p(\mathbb B_n) \ \text{if and only if} \  R^Nf\in  A_{(N-\beta)p+\alpha}^p(\mathbb B_n).
\]
Therefore, let $X(\mathbb B_n) =A_{\alpha,\beta}^p(\mathbb B_n)$ and $Y(\mathbb B_n) =A_{(N-\beta)p+\alpha}^p(\mathbb B_n)$, where $N=\inf\{\hat{N}\in\mathbb Z_{\geq 1}:\hat{N}>\beta -\frac{\alpha+\frac{1}{2}}{p}\}$. Moreover, for $f\in \HH (\mathbb B_n)$ we have
\begin{equation}\label{BergSobNormEquivRel}
\norm{(I+R)^\beta f}_{A_\alpha^p} \asymp \abs{f(0)}+\norm{R^\beta f}_{A_\alpha^p},
\end{equation}
according to Lemma \ref{BergSobEquivNorm}. Condition \ref{RelationXtoY} follows by first using the equivalence of the norms $\norm{\cdot}_{A_{\alpha,\beta}^p}$ and $\norm{\cdot}_{A_{(N-\beta)p+\alpha,N}^p}$ by (\ref{formimp}), and then applying (\ref{BergSobNormEquivRel}) to the latter norm. 
Furthermore, it holds that $M(A^p_\alpha(\mathbb B_n)) = H^\infty(\mathbb B_n)$, which shows that condition \ref{HinftySubsetMultiplY} is satisfied.

Let us check the condition \ref{divideByPolynom} for $A_{\alpha,\beta}^p(\D)$. We assume that $f \in A_{\alpha,\beta}^p(\D)$ has a zero at $z=z_0$. Let us show that $\frac{f}{z-z_0} \in  A_{\alpha,\beta}^p(\D)$  by establishing that $R^N\left(\frac{f}{z-z_0}\right)  \in A^p_{(N-\beta)p + \alpha}(\D)$.   
Let us take $|z_0| < r < 1.$ We may assume that $|z| \ge r,$ since $R^N\left(\frac{f}{z-z_0}\right) \in H(\D)$ is bounded on $r\D.$ We will utilize the following formula given in Proposition 6 in \cite{CHZ}:
\begin{align*}
R^N\left(\frac{f(z)}{z-z_0}\right) = \frac{(-1)^N}{(z-z_0)^{N+1}} \sum_{k=0}^N (-1)^k \binom{N+1}{k} (z-z_0)^k R^N((z-z_0)^{N-k}f),
\end{align*}
where $r \le |z| < 1$. It suffices to show that $R^N((z-z_0)f) \in A^p_{(N-\beta)p + \alpha}(\D)$, which implies that $R^N((z-z_0)^{N-k}f) \in A^p_{(N-\beta)p + \alpha}(\D)$ for $k=0,1,\ldots,N$. Using the general Leibniz rule we obtain 
\begin{align*}
R^N((z-z_0)f) = \sum_{k=0}^N \binom{N}{k} R^{N-k}(z-z_0)R^k(f) =  \sum_{k=0}^{N} \binom{N}{k} z R^k(f).
\end{align*}
We observe that $\|zR^kf\|_{A^p_{(N-\beta)p + \alpha}} \le \|R^kf\|_{A^p_{(N-\beta)p + \alpha}}$
and $R^kf \in A^p_{(N-\beta)p + \alpha}(\D)$ if and only if 
\[
(1-|z|^2)^{N-k}R^{(N-k)}R^kf = (1-|z|^2)^{N-k}R^{N}f \in A^p_{(N-\beta)p + \alpha}(\D),
\] 
see \cite[p.~75]{Z}. The last statement holds, since 
\[
\norm{(1-|z|^2)^{N-k}R^{N}f }_{A^p_{(N-\beta)p + \alpha}} \le \norm{R^{N}f}_{A^p_{(N-\beta)p + \alpha}} < \infty,
\] 
where we used the fact $R^Nf  \in A^p_{(N-\beta)p + \alpha}(\D).$ So we have that $R^kf \in A^p_{(N-\beta)p + \alpha}(\D)$ for $k = 0,1,\ldots,N$ and consequently  $R^N((z-z_0)f) \in A^p_{(N-\beta)p + \alpha}(\D)$. Therefore $R^N\left(\frac{f}{z-z_0}\right)  \in A^p_{(N-\beta)p + \alpha}(\D).$ 

It should also be mentioned that the spaces $A^p_{\alpha,\beta}(\mathbb B_n)$ are reflexive for $p>1$, see Proposition 5.7 (iv) in \cite{BB}.

\smallskip

$\textbf{(c)}$ 
We consider all spaces $X(\mathbb B_n)$  that satisfies \ref{finerThanCOTop} and $M(X(\mathbb B_n))=H^\infty(\mathbb B_n)$. Furthermore, condition \ref{divideByPolynom} is also assumed to hold if $n=1$. Letting $Y(\mathbb B_n)=X(\mathbb B_n)$ condition \ref{HinftySubsetMultiplY} is also satisfied and condition \ref{RelationXtoY} is irrelevant, see the remark after condition \ref{RelationXtoY}. These spaces include growth spaces $H^\infty_\alpha(\D), \, \alpha  >0,$ and weighted Hardy spaces $H_w^p(\mathbb D),\ p >1$, where $w\in (A^p)$, that is, $w$ satisfies the Muckenhoupt $(A^p)$-condition, see details in \cite{BGZ}. Considering the weighted Hardy spaces, condition  \ref{finerThanCOTop} follows from the proof of Lemma 2.1 in \cite{BGZ}. Notice that if $w\in (A^p), \ p>1$, then the critical exponent $q_w<p$. For $f\in\mathcal H (\mathbb D)$ we have $\norm{f}_{H_w^p}<\infty$ if and only if 
\[
\lim_{r\to1^-} \int_{-\pi}^\pi \left| f(re^{i\theta}) \right|^p w(\theta) d\theta   <\infty.
\]
Since, for every $z_0\in\mathbb D$, there exists $r<1$ such that $\frac{1}{z-z_0}$ is bounded on $\mathbb D\setminus r\mathbb D$, condition \ref{divideByPolynom} follows. Condition \ref{divideByPolynom} is proved by similar arguments for many spaces, for example, weighted Bergman spaces, growth spaces and Hardy spaces.

\begin{lem}\label{BergSobEquivNorm}
Let $\beta\geq 0$. If either $\alpha> -1$ and $p\geq 1$, or $\alpha= -1$ and $p=2$, it holds that

\[
\norm{f}_{A^p_{\alpha,\beta}} \asymp \abs{f(0)}+\norm{R^\beta f}_{A^p_\alpha} , \ f\in\mathcal H (\mathbb B_n).
\]
Moreover, the space $A^p_{\alpha,\beta}(\mathbb B_n)$ endowed with the norm defined as $\norm{f}_{p,\alpha,\beta}=\abs{f(0)}+\norm{R^\beta f}_{A^p_\alpha}$ is a Banach space.
\end{lem}
\begin{proof}
Let $N$ be the smallest integer in the set $\mathbb Z_{> \beta+1}$ and $\gamma=p(N-\beta)+\alpha>-1$.   By (\ref{formimp}) and Lemma 1 in \cite{CHZ} we have $\norm{f}_{A^p_{\alpha,\beta}} \asymp\norm{f}_{A^p_{\gamma,N}} $ and
$\norm{f}_{p,\alpha, \beta} \asymp \norm{f}_{p, \gamma, N}$
 respectively. The norm equivalences are also well known to experts in the case $\alpha=-1$ and $p = 2$.  Therefore 
\[
\left(A^p_{\alpha,\beta}(\mathbb B_n),   \norm{\cdot}_{p,\alpha,\beta}   \right)
\] 
is a Banach space, since this is true for 
\[
\left(A^p_{\alpha,\beta}(\mathbb B_n), \norm{\cdot}_{p,\gamma',N}\right)
\]
for all $\gamma'>-1$,
see \cite{ZZ}. 
It now suffices to show that $\norm{f}_{A^p_{\gamma,N}}  \asymp  \norm{f}_{p, \gamma, N}$. Using Jensen's inequality we have
\begin{align*}
\left|(I+R)^N f(z)\right| &=\left|\sum_{j=0}^N \binom{N}{j} R^j f(z)\right|\\
& \leq2^N\sum_{j=0}^N \frac{1}{2^N}\binom{N}{j} \left|R^j f(z)\right|\\
& \leq 2^N \left(\sum_{j=0}^N \frac{1}{2^N}\binom{N}{j} \left|R^j f(z)\right|^p\right)^\frac{1}{p}.
\end{align*}
Furthermore,
{
\allowdisplaybreaks
\begin{align*}
\norm{f}_{A^p_{\gamma, N}}&=\left(\int_{\mathbb B _n}\left|(I+R)^N f(z)\right|^p dA_\gamma (z)\right)^\frac{1}{p}\\
&\leq \left( 2^{N(p-1)} \sum_{j=0}^N \binom{N}{j}\int_{\mathbb B _n} \left|R^j f(z)\right|^p dA_\gamma(z)\right)^\frac{1}{p}\\
&\asymp   \left(\sum_{j=0}^N\int_{\mathbb B _n} \left|R^N f(z)\right|^p dA_{p(N-j)+\gamma}(z)\right)^\frac{1}{p}\\
&\lesssim \left(  \int_{\mathbb B _n} \left|R^{N} f(z)\right|^p   dA_\gamma(z)\right)^\frac{1}{p} +\abs{f(0)} =  \norm{f}_{p, \gamma, N}. 
\end{align*}
}Our approach to prove the converse is very similar. It holds that
{
\allowdisplaybreaks
\begin{align*}
\int_{\mathbb B _n}\left|R^N f(z)\right|^p dA_\gamma (z)&=\int_{\mathbb B _n}\left|(I+R-I)^N f(z)\right|^p dA_\gamma (z)\\
&\leq 2^{N(p-1)} \sum_{j=0}^N \binom{N}{j}\int_{\mathbb B_n} \left|(I+R)^j f(z)\right|^p dA_\gamma(z)\\
&\asymp  \sum_{j=0}^N \int_{\mathbb B _n} \left|(I+R)^{N} f(z)\right|^p   dA_{p(N-j)+\gamma}(z)\\
&\lesssim \int_{\mathbb B _n} \left|(I+R)^N f(z)\right|^p   dA_\gamma(z).
\end{align*}
}From this and Lemma 5.6 in \cite{BB}, it follows that there exists a constant $M=M(n,N,p,\gamma)>0$ such that
\begin{align*}
\norm{R^Nf}_{A^p_\gamma} + \abs{f(0)} &\leq M\norm{(I+R)^Nf}_{A^p_\gamma}+\abs{f(0)}\leq 2M\norm{(I+R)^Nf}_{A^p_\gamma},
\end{align*}
which finishes the proof.
\end{proof}

\section{The spectrum of $M_u$}

\bigskip

Next, we will characterize the spectra of multiplication operators acting on $X(\mathbb B_n)$ in the case that there exists a space $Y(\mathbb B_n)$ such that conditions \ref{finerThanCOTop} - \ref{HinftySubsetMultiplY} are satisfied.
Condition \ref{RelationXtoY} is crucial in the following lemma. The corresponding results for the Hardy-Sobolev Hilbert spaces were obtained in \cite{CHZ}.
\begin{lem}\label{MuInvert} Assume that \emph{\ref{finerThanCOTop}}, \emph{\ref{RelationXtoY}} and \emph{\ref{HinftySubsetMultiplY}} are satisfied and let  $u\in M(X(\mathbb B_n))$. The following statements are equivalent:
 \smallskip
\begin{enumerate}[label=\upshape(\alph*)]
\itemsep-0.6em 
\item $\frac{1}{u}\in M(X(\mathbb B_n))$,\label{reciprocal_uMX}\\
\item $\frac{1}{u}\in H^\infty \mathbb (\mathbb B_n)$,\label{reciprocal_uHinfty}\\
\item $M_u$ is invertible.\label{muInvert}
\end{enumerate}

\end{lem}

\begin{proof}
Assuming $\frac{1}{u}\in M(X(\mathbb B_n))$ we obtain immediately, by the remark after condition  \ref{HinftySubsetMultiplY}, that $\frac{1}{u}\in H^\infty (\mathbb B_n)$. To prove the converse implication we will use the formula
\begin{equation}\label{RadialDerivOfQuotient}
R^N \left( \frac{f}{u} \right) = \frac{  (-1)^N   }{   u^{N+1}  } \sum_{k=0}^N (-1)^k \binom{N+1}{k} u^k  R^N \left( u^{N-k} f  \right), \ f\in \HH(\mathbb B_n),
\end{equation}
which can  be found in Corollary 5 in \cite{CHZ}. 
The proof of the formula uses the derivative $D$, but the formula remains valid for all linear operators $S$ that admit the law $S(fg) = fSg+gSf, \ f,g\in\mathcal H (\mathbb D),$ and for which the formula is valid for $N=1$. Moreover, the dimension $n$ is irrelevant for the proof, and therefore we may replace $D$ with $R$ and also consider the formula in higher dimensions. Notice that (\ref{RadialDerivOfQuotient}) is invalid for $N=0$. 

 If  $\frac{1}{u}\in H^\infty(\mathbb B_n)$, then $u$ is uniformly bounded from below, that is, there exists a $0<c<1$ such that $\inf_{z\in\mathbb  B_n} |u(z)|\geq c$ and hence formula (\ref{RadialDerivOfQuotient})  is applicable. By condition \ref{RelationXtoY}  we have that $f\in X(\mathbb B_n)$ if and only if
 $R^N f  \in Y(\mathbb B_n)$. 

One should also notice that $u^k\in M(X(\mathbb B_n))$ and $(\frac{1}{u})^k\in H^\infty(\mathbb B_n)$ for all $k\in\mathbb Z_{\geq 0}$. For $f\in X(\mathbb B_n)$ we obtain that 
$u^kf\in X(\mathbb B_n)$ for all $k\in\mathbb Z_{\geq 0}$, and therefore,

{
\allowdisplaybreaks
\begin{align*}
\norm{   \frac{f}{u}  }_X&\asymp \norm{   R^N \left( \frac{f}{u} \right)  }_Y +\abs{\frac{f(0)}{u(0)}}\\
&\leq  \sum_{k=0}^N  \binom{N+1}{k}  \norm{  \left(   \frac{1}{u}   \right)^{N+1-k}  R^N \left( u^{N-k} f  \right) }_Y+\abs{\frac{f(0)}{u(0)}}\\
&\lesssim   \sum_{k=0}^N   \norm{  \left(   \frac{1}{u}   \right)^{N+1-k} }_\infty \norm{ R^N \left( u^{N-k} f  \right) }_Y+\abs{\frac{f(0)}{u(0)}}\\
&\leq  \left(   \frac{1}{c}   \right)^{N+1}  \sum_{k=0}^N \norm{ R^N \left( u^{N-k} f  \right) }_Y+\abs{\frac{f(0)}{u(0)}}\\
&\lesssim  \sum_{k=0}^N   \norm{  u^{N-k} f  }_X+\abs{\frac{f(0)}{u(0)}} < \infty,
\end{align*}
}where the remark after condition \ref{HinftySubsetMultiplY} gives the second inequality. Hence, we have shown that the two statements \ref{reciprocal_uMX} and \ref{reciprocal_uHinfty} are equivalent. If $\frac{1}{u}\in  M(X(\mathbb B_n))$, then clearly $f\mapsto \frac{f}{u}$ is the inverse of $M_u$.  Conversely, if $M_u$ is invertible, then $M_{\frac{1}{u}}$ must be the unique bounded inverse, so 
$\frac{1}{u}\in M(X(\mathbb B_n)).$
\end{proof}

\begin{thm}\label{MainThmSpec}
Assume that \emph{\ref{finerThanCOTop}}, \emph{\ref{RelationXtoY}} and \emph{\ref{HinftySubsetMultiplY}} are satisfied and let $M_u:X(\mathbb B_n) \to X(\mathbb B_n)$ be a multiplication operator generated by  $u\in M(X(\mathbb B_n))$. The spectrum of $M_u$ is given by $\sigma(M_u) =
 \closed{u(\mathbb B_n)}$.
\end{thm}
\begin{proof}Let $\lambda\in\mathbb C$. Clearly  $u - \lambda\in  M(X(\mathbb B_n)).$   If $\lambda\in \closed{u(\mathbb B_n)}$, then $|u(z)-\lambda|$ is not bounded from below so  
$M_u-\lambda I = M_{u-\lambda}$ is not invertible by Lemma  \ref{MuInvert}. Using again Lemma \ref{MuInvert}, it
follows that for any $\lambda\in\mathbb{C} \setminus\closed{u(\mathbb B_n)}$ the operator $M_u-\lambda I $ is invertible since $|u(z)-\lambda|$, in this case, is bounded from below. Hence, the spectrum is given by $\sigma(M_u) = \closed{u(\mathbb B_n)}$.
\end{proof}

\begin{rem}  The above result implies that $r(M_u) = ||u||_\infty\le ||M_u||$. Moreover, since the spectrum $\sigma(M_u) =
 \closed{u(\mathbb B_n)}$ is connected, when $u$ is continuous, any nonzero spectral radius would imply an uncountable number of points in the spectrum, from which it follows that the operator is not compact. Consequently, $M_u$ is never compact if $u\ne 0.$
\end{rem}

\begin{cor}\label{examplesOfSpectr}
Let  $X(\mathbb B _n)$  be  any of the following spaces
\begin{enumerate}[label=\upshape(\alph*)]
\itemsep-0.6em 
\item $A^p_{\alpha,\beta}(\mathbb B _n),\ p\geq 1, \ \beta\geq 0, \text{ and }\alpha>-1$;\\
\item $\mathcal B_\alpha (\mathbb B _n),\ \alpha >0$;\\
\item $H_\beta^2(\mathbb B _n),\ \beta\geq 0$;\\
\item $H_w^p(\mathbb D),\ p >1, \ w \in (A^p)$.
\end{enumerate}
Then the spectrum of a multiplication operator $M_u:X(\mathbb B _n)\to X(\mathbb B _n)$ is given by $\sigma(M_u) = \closed{u(\mathbb B_n)}$.
\end{cor}

\bigskip

\section{The essential spectrum of $M_u$}\label{sect4}

\bigskip

Examining the essential spectrum of a multiplication operator when the domain is $\mathbb B_n, \ n>1$, the result concerning $H_\beta^2$, obtained by Cao, He and Zhu, can be made quite general, see Theorem \ref{EssSpecForMultiDim}. In the case $n=1$, we have obtained a sufficient condition for Fredholmness in Lemma \ref{ThenFredholm_ForX}, where all four conditions \ref{finerThanCOTop}-\ref{divideByPolynom} were assumed. For the spaces mentioned in our main result, namely Theorem \ref{MainResEssSpec}, this condition is also necessary for Fredholmness, see lemmas \ref{IfFredholm_BlochType} and \ref{IfFredholm_BergmanSobolev}, but for this to be proved, space-specific properties were used. An asymptotic approximation for the behaviour of the norm of the peak functions is necessary for the result concerning Bergman-Sobolev spaces. The estimate given in Lemma 11 in \cite{CHZ} is insufficient for our purposes, not only because it only considers $p=2$, but also because it is not a sharp lower bound. The necessity of an asymptotic approximation instead of a non-sharp lower bound of the behaviour is clear when an arbitrary $p\in(1,\infty)$ is considered in Theorem \ref{IfFredholm_BergmanSobolev}.

\begin{thm}\label{EssSpecForMultiDim} Assume that condition \emph{\ref{finerThanCOTop}} is satisfied and $n >1$. Furthermore, let $u\in M(X(\mathbb B_n))$  and $P_j\colon \mathbb B_n \to \C, \, P_j(z) = z_j$ for every $j = 1, \ldots , n.$ Suppose that $P_j \in  M(X(\mathbb B_n))$ for every $j$. 

Then $  \sigma_e(M_u)  = \bigcap_{0 < r < 1} \closed{u(\mathbb B_n \setminus r\mathbb B_n)} = \closed{u(\mathbb B_n)}  = \sigma(M_u) $.
\end{thm}

\begin{proof}  Let $\lambda\in u(\mathbb B_n).$ Since $n >1$, the function $u(z)  - \lambda$ has infinitely many  distinct zeros, and therefore, there must exist an infinite subset $\{\alpha_k\}_{k=1}^\infty, \ \alpha_k=(\alpha_{k,1},...,\alpha_{k,n}),$ of these zeros such that for some $j = 1,\ldots,n$ we have $\alpha_{k,j} \ne \alpha_{l,j}$ whenever $k \ne l$.  
We first show, by induction, that $(\delta_{\alpha_k})_{k=1}^\infty$  are  linearly independent in $Ker \ M^*_{u-\lambda}$. Clearly all  $\delta_{\alpha_k}\in Ker \ M^*_{u-\lambda}.$ Suppose that
\[
\sum_{k=1}^m c_k \delta_{\alpha_k}=0
\]
for some $m\in\mathbb Z_{\geq 1}$. If $m=1$, it follows that $c_1=0$. Assume that $m\geq 2$. For arbitrary $f\in X(\mathbb B_n)$ we have by assumption that  $P_j f\in X(\mathbb B_n)$, so
\[
\sum_{k=1}^m c_k \alpha_{k,j} \delta_{\alpha_k}(f)=0  \ \text{and}\ \sum_{k=1}^m c_k \delta_{\alpha_k}(f)=0.
\]
Hence 
\[
\sum_{k=2}^m c_k (\alpha_{k,j} - \alpha_{1,j})\delta_{\alpha_k}(f)= \sum_{k=1}^m c_k (\alpha_{k,j} - \alpha_{1,j})\delta_{\alpha_k}(f)=0 \ \text{for all} \  f \in  X(\mathbb B_n),
\]
and therefore, by the induction hypothesis, $c_k(\alpha_{k,j} - \alpha_{1,j})=0$ for all $k=2,...,m.$  This implies that $c_k=0$ for $k=2,...,m,$ and consequently $c_1=0.$
Then   $Ker \ M^*_{u-\lambda}$ is infinite dimensional so that $M^*_{u-\lambda}$, and equivalently $M_{u-\lambda}$, is not Fredholm. It follows that $u(\mathbb B_n)\subset \sigma_e(M_u)$ and, moreover, that
\[
 \bigcap_{0 < r < 1} \closed{u(\mathbb B_n \setminus r\mathbb B_n)}\subset \closed{u(\mathbb B_n)} \subset  \sigma_e(M_u)\subset \sigma(M_u).
\]
For the converse conclusion,
let      $\lambda \notin \bigcap_{0 < r < 1} \closed{u(\mathbb B_n \setminus r\mathbb B_n)}.$ Hence, there are $r\in (0,1)$ and $\delta>0$ such that $|\lambda - u(z)| \ge \delta$ for all
$r < |z| < 1.$ Then $v(z) = (u(z) - \lambda)^{-1}$ is holomorphic and bounded on $\mathbb B_n\setminus r\overline{\mathbb B}_n.$  As in \cite{CHZ}, using Hartogs' extension theorem and the identity theorem,   we can extend $v$  to a function  $\tilde v\in\HH(\mathbb B_n)$ such that 
$\tilde v(z) = (u(z) - \lambda)^{-1}$ for all  $z\in \mathbb B_n,$ and therefore $\tilde v\in H^\infty(\mathbb B_n).$ Now  $M_{u-\lambda}$ is invertible by
 Lemma \ref{MuInvert}, so $\lambda\notin\sigma(M_u).$
\end{proof}

\begin{rem}\label{nOneIndependentPointEval}
Following the proof of Theorem \ref{EssSpecForMultiDim} it is clear that $(\delta_{\alpha_k})_{k=1}^\infty$ are linearly independent when $n=1$.
\end{rem}
Now we proceed to the case $n=1$.

The following result is based on ideas due to Axler \cite{A} that was  carried on in \cite{BDL}. It holds for all  spaces $X(\mathbb D)$  such that 
$M(X(\mathbb D)) = H^\infty(\mathbb D).$  

\begin{lem}\label{IfFredholm_ForX} Assume that condition \emph{\ref{finerThanCOTop}} is satisfied and let $u\in M(X(\mathbb D))=H^\infty(\mathbb D)$. If $M_u \colon X(\mathbb D)\to X(\mathbb D)$ is Fredholm, then there are $r\in (0,1)$ and $\delta>0$ such that $|u(z)| \ge \delta$ for all 
$r\le |z| < 1.$
\end{lem}

\begin{proof}  Assume we can find  a sequence $(z_n)_{n=1}^\infty \subset { \mathbb D}$  with $|z_n|\to 1$ and $|u(z_n)|\to 0$  when $n\to\infty.$   Then we can assume that $(z_n)_n$ is an interpolating sequence in $H^\infty(\mathbb D)$ by going to a subsequence if necessary. Therefore, (see e.g. \cite{An}, Ch. 7.3) there is a constant $M> 0$ such that for  each $N\in\mathbb N$ there is a function $u_N\in H^\infty(\mathbb D)$ with
\begin{equation*}
u_N(z_n) = \left\{
\begin{array}{@{}ll}
u(z_n), &  n\ge N,\\
0, &  n < N
\end{array} \right.
\end{equation*}
and $\norm{u_N}_\infty \le M \sup_{n\ge N} \abs{u(z_n)}$. Let
\[
Z_N = \{ f \in X(\mathbb D): \delta_{z_n}(f) = 0 \ \text{for all} \ n\ge N\},
\]
which is a closed subspace of $X(\mathbb D)$. From Remark \ref{nOneIndependentPointEval} we know that the $\delta_{z_n}\in X(\mathbb D)^*$   are  linearly independent, which implies that $Z_N ^\perp$ is infinite-dimensional. Since $\delta_{z_n}(u- u_N) =0$ for all $n\ge N$, we get $M_{u- u_N}(X(\mathbb D))\subset Z_N.$ Now  $(X(\mathbb D)/Z_N)^* = Z_N^\perp$, so $X(\mathbb D)/Z_N$  is infinite-dimensional.
Hence $X(\mathbb D)/ M_{u- u_N}(X(\mathbb D))$  is also infinite-dimensional, and $M_{u -u_N}\colon X(\mathbb D) \to X(\mathbb D)$ is not Fredholm. As
 $M(X(\mathbb D)) = H^\infty(\mathbb D)$  and the set of non-Fredholm operators is closed, it follows from 
\[
\norm{M_{u - u_N} - M_u} = \norm{M_{u_N}}\le C \norm{u_N}_\infty\le C M \sup_{n\ge N} \abs{u(z_n)}\to 0 \ \text{as} \ N\to\infty,
\]
that  $M_u$ is not Fredholm.
\end{proof}

\begin{lem}\label{ThenFredholm_ForX}  
Assume that \emph{\ref{finerThanCOTop}}, \emph{\ref{RelationXtoY}}, \emph{\ref{HinftySubsetMultiplY}} and \emph{\ref{divideByPolynom}} are satisfied and let  $u\in M(X(\mathbb D))$. If  there are $r\in (0,1)$ and $\delta>0$ such that $|u(z)| \ge \delta$ for all $r\le |z| < 1$,  then $M_u \colon X(\mathbb D)\to X(\mathbb D)$  is Fredholm.
\end{lem}

\begin{proof}  By assumption we have that $u$  can have only finitely many zeros $\alpha_1,...\alpha_n$ inside $\mathbb D$ with multiplicities $m_1,...,m_n$ respectively.
Then for all $z\in\mathbb D$,
\[
u(z) = v(z) (z - \alpha_1)^{m_1}... (z - \alpha_n)^{m_n}= v(z) p(z),
\]
where $v\in\HH(\mathbb D)$ and $\frac{1}{v}\in H^\infty(\mathbb D)$. 

Let us now define the point evaluation  maps for derivatives by  $\delta^{(k)}_z(f) = f^{(k)}(z)$ for   all $z\in \mathbb D$ and all $k \in\mathbb Z_{\geq 0}$. By assumption \ref{finerThanCOTop}, it holds that $\delta^{(k)} _z  \in X(\mathbb D)^*$ for all $k$ and $z$. Clearly,
\[
M_u(X(\mathbb D)) \subset \bigcap_{i=1}^n\bigcap _{k=0}^{m_i-1} {\text  Ker} \ \delta_{\alpha_i}^{(k)}.
\]
Let $f \in  \bigcap_{i=1}^n\bigcap _{k=0}^{m_i-1} {\text  Ker} \ \delta_{\alpha_i}^{(k)},$  so  $ f^{(k)}(\alpha_i)=0$ for   all $i=1,...,n$ and  all $k=0,...,{m_i}-1.$ Then $\frac{f}{u} \in
 \HH(\mathbb D).$ Now assumption \ref{divideByPolynom} implies that $v\in M(X(\mathbb D)).$ Indeed, if $g\in X(\mathbb D)$, then $ug\in X(\mathbb D)$ and by assumption \ref{divideByPolynom} it follows  that $v g = \frac{ug}{p} \in X(\mathbb D).$  Therefore, $\frac{1}{v} \in M(X(\mathbb D))$ by Lemma \ref{MuInvert}, so that $\frac{f}{u} = \frac{f/p}{v}\in X(\mathbb D)$ by assumption \ref{divideByPolynom}. As a result, $f=u\frac{f}{u}\in M_u(X(\mathbb D))$, and thus
\[
M_u(X(\mathbb D)) = \bigcap_{i=1}^n\bigcap _{k=0}^{m_i-1} \text{ Ker} \ \delta_{\alpha_i}^{(k)}.
\]
Consequently, $M_u$ has closed range, and  since  $M_u\colon  X(\mathbb D)\to X(\mathbb D)$  is always  injective, the dimension of the kernel of $M_u$ is finite.
Since $^\perp(\spa\{\delta_{\alpha_i}^{(k)}\}) = {\text Ker} \  \delta_{\alpha_i}^{(k)}$, it follows that  the $w^*$-closed  one-dimensional space  
$\spa\{\delta_{\alpha_i}^{(k)}\}  = ( {\text  Ker} \ \delta_{\alpha_i}^{(k)})^\perp$, see \cite[Theorem 11 on p.~341]{M}. Therefore, by \cite[Theorem 13 on p.~342]{M}, we have
\[
M_u(X(\mathbb D))^\perp = \sum_{i=1}^n\sum _{k=0}^{m_i-1} (\text{ Ker} \ \delta_{\alpha_i}^{(k)})^\perp =\sum_{i=1}^n\sum _{k=0}^{m_i-1} \spa
\{ \delta_{\alpha_i}^{(k)} \},
\]
and hence, the dimension of the co-kernel of $M_u$ is finite, and $M_u$ is Fredholm.
\end{proof}

\begin{thm}\label{EssSpecMlutiHInfty}  Assume that \emph{\ref{finerThanCOTop}}, \emph{\ref{HinftySubsetMultiplY}} and \emph{\ref{divideByPolynom}} are satisfied and 
$M(X(\mathbb D)) = H^\infty(\mathbb D).$ Let $M_u\colon X(\mathbb D) \to X(\mathbb D)$ be a multiplication operator generated by  $u\in M(X(\mathbb D))$. 
The essential spectrum of $M_u$ is given by 
$\sigma_e(M_u) = \bigcap_{0 < r < 1} \closed{u(\mathbb D \setminus r\mathbb D)}$.
\end{thm}
\begin{proof}  We have that  $\lambda\in\closed{u(\mathbb D \setminus r\mathbb D)}$ for all $r\in(0,1)$  if and only if for all  $r\in(0,1)$ there is a sequence $(z_n)_{n=1}^\infty\subset\mathbb D$ 
such that $|z_n| \ge r$ for all $n\in\mathbb N$ and $|u(z_n) -\lambda| \to 0$ when $n \to\infty.$ Since $M_u-\lambda I = M_{u-\lambda}$, we can now apply lemmas \ref{IfFredholm_ForX} and \ref{ThenFredholm_ForX} to conclude that the last statement equivalently  means that   $M_u-\lambda I$ is not Fredholm, that is $\lambda\in\sigma_e(M_u). $ The use of Lemma \ref{ThenFredholm_ForX} is justified by the remark after condition {\ref{RelationXtoY}}.
\end{proof}

In {\bf  Examples} it was stated that the multiplier spaces for $A_{\alpha,\beta}^p(\D)$ with $p \geq 1, \ \alpha>-1 , \ \beta < \frac{1+\alpha}{p} $; $H_w^p(\D)$ with $p>1, \ w\in (A^p)$ and $\B_\alpha(\D)$ with $\alpha>1$ are $H^\infty(\D)$. Thus, we obtain the following results.

\begin{cor}\label{exampleOfEssSpectr}
In each of the following three cases:
\begin{enumerate}[label=\upshape(\alph*)]
\itemsep-0.6em 
\item $p\geq 1,\ \alpha>-1$ and $\beta < \frac{1+\alpha}{p}$ with $u\in M(A^p_{\alpha,\beta}(\D))$;\\
\item $\alpha>1$ with $u\in M(\B_\alpha(\D)  )$;\\
\item $p >1, \ w \in (A^p)$ with $u\in M(H_w^p(\D))$,
\end{enumerate}
the essential spectrum of $M_u$ is given by
\[
\sigma_e(M_u) = \bigcap_{0 < r < 1} \closed{u(\mathbb D \setminus r\mathbb D)}.
\]
\end{cor}

It was shown in Theorem \ref{EssSpecForMultiDim} that in higher dimensions, $n>1$, the essential spectra of multiplication operators coincide with their spectra for many spaces. This is seldom true for $n=1$. In corollaries \ref{examplesOfSpectr} and \ref{exampleOfEssSpectr} and Theorem \ref{MainResEssSpec}, we list some spaces, on which multiplier operators have the spectrum given by $ \closed{u(\mathbb B_n)}$ and the essential spectrum given by $\bigcap_{0 < r < 1} \closed{u(\mathbb B_n \setminus r\mathbb B_n)} $. Although the sets may differ, their spectral and essential spectral radii coincide according to the following remark.

\begin{rem} (a) Let $n\in\mathbb Z_{\geq 1}$. Since the decreasing sequence $\left(\closed{u(\mathbb B_n \setminus \left(1-\frac{1}{k})\mathbb B_n\right)}\right)_{k=2}^\infty$ consists of compact and connected sets, the intersection $\bigcap_{0 < r < 1} \closed{u(\mathbb B_n \setminus r\mathbb B_n)}$ is compact and connected.\\\\
(b) For  $n\in \mathbb Z_{\geq 1}$ and $u\in H^\infty (\mathbb B_n)$ we have 
\[
\sup \left\{  \abs{\lambda} : \lambda \in \bigcap_{0 < r < 1} \closed{u(\mathbb B_n \setminus r\mathbb B_n)} \right\}=  \norm{u}_\infty =\sup_{\lambda \in \closed{u(\mathbb B_n)}}    \abs{\lambda}. 
\]
Moreover, both suprema are attained. Clearly 
\[
\norm{u}_\infty = \sup_{z\in\mathbb B_n} \abs{u(z)} =  \sup_{\lambda\in u(\mathbb B_n)} \abs{\lambda} =  \sup_{\lambda\in \closed{u(\mathbb B_n)}} \abs{\lambda}\geq \sup \left\{  \abs{\lambda} : \lambda \in \bigcap_{0 < r < 1} \closed{u(\mathbb B_n \setminus r\mathbb B_n)} \right\}. 
\]
Furthermore, since $u\in H^\infty$ there is a sequence $(z_j)_{j=1}^\infty$ such that $z_j\in \mathbb B_n\setminus r_j\mathbb B_n$ and $\lim_{j\to\infty} \abs{u(z_j)}= \norm{u}_{\infty}$, where $r_j=1-j^{-1}$. The sequence $(u(z_j))_{j=1}^\infty$ is bounded, and therefore, by Bolzano - Weierstrass theorem, there is a convergent subsequence $(\lambda_k)_{k=1}^\infty$, where $\lambda_k = u(z_{j_k}) \in \closed{u(\mathbb B_n\setminus r_{j_k}\mathbb B_n)}$.  Since  the sets $U_k = \closed{u(\mathbb B_n\setminus r_{j_k}\mathbb B_n)}$ are compact and $U_{k+1}\subset U_k, \ k=1,2,...$, it holds that $\lim_{k\to\infty}\lambda_k = \lambda\in U_j$ for every $j$, and hence, we have $\lambda \in \bigcap_{0 < r < 1} \closed{u(\mathbb B_n \setminus r\mathbb B_n)}$ and $\abs{\lambda} = \norm{u}_\infty$.

\end{rem}

\bigskip

For $\xi\in\partial \mathbb D$ and $k\in\mathbb Z_{\geq 1}$, let $f_{\xi,k}:\mathbb D\to\mathbb D$ be a peak function defined by
\[
f_{\xi,k}(z) = \left(\frac{1+\conj{\xi}z}{2}\right)^k.
\]

For  $\alpha > 0 $ it is well-known that  
$\B_{0,\alpha}(\D)^{*} \simeq { A^1_0}(\mathbb D)$ and   ${A^1_0}(\mathbb D)^{*} \simeq \B_\alpha(\D)$ via an integral pairing, see \cite{Z}.

\begin{lem} \label{usefulBloch}  Let  $0< \alpha \le 1$, $\xi\in\D$, and $g_{\xi,k}(z)= \left( \frac{1+\conj{\xi}z}{2} \right)^k\norm{\left( \frac{1+\conj{\xi}z}{2} \right)^k}_{\B_\alpha}^{-1}$ be the normalized peak function.
Then we have $g^{(m)}_{\xi,k}\to 0, \ m\in\mathbb Z_{\geq 0}$ uniformly on every set $A_\delta= \{z\in\D: |z -\xi| \ge \delta\}$, $\delta>0,$ and $g_{\xi, k} \to  0$ weakly in $\B_\alpha(\D)$ as $k\to \infty$.
\end{lem}

\begin{proof}   For the Bloch-type spaces  $\B_\alpha(\D)$,  it can be shown that 
\[
 \norm{\left( \frac{1+\conj{\xi}z}{2} \right)^k}_{\B_\alpha} \asymp k^{1-\alpha}.
\]
The property $g^{(m)}_{\xi,k}\to 0, \ m\in\mathbb Z_{\geq 0},$ uniformly 
on the  sets $A_\delta$  as $k\to \infty$ is a consequence of the definition of $g_{\xi,k}$. Moreover, the  sequence $(g_{\xi,k})_{k=1}^\infty$ is a $weak^*$ null sequence  by  using Lemma 3.1 in \cite{CPPR}. Since $(g_{\xi,k})_k \subset \mathcal P(\mathbb D)\subset \B_{0,\alpha}(\mathbb D)$,  we conclude that $g_{\xi, k}\to 0$ weakly when $k\to\infty$.

\end{proof}

\begin{lem}\label{IfFredholm_BlochType}
 Let us assume that either $u\in M(\B(\D))\cap  A(\mathbb D)$   or $u\in M(\B_\alpha(\D))=\mathcal B_\alpha(\D)$ with $0 < \alpha <1.$
If $M_u\colon \B_\alpha(\D) \to \B_\alpha(\D)$ is Fredholm, then there are $r\in (0,1)$ and $\delta>0$ such that $|u(z)| \ge \delta$ for all $r\le |z| < 1.$
\end{lem}

\begin{proof} Suppose there  is a sequence $(z_k)_{k=1}^\infty \subset\mathbb D$ such that $|z_k|\to 1$ and $|u(z_k)| \to 0$ when $k\to\infty.$ Then, by going to a subsequence if necessary, we can assume that $z_k\to\xi\in\partial\mathbb D$ when $k\to\infty.$  Since $u$ is continuous  up to the boundary in both cases, $u(\xi)=0.$ Now by  Lemma \ref{usefulBloch} it holds that $g_{\xi,k}\to 0,$ $g'_{\xi,k}\to 0$ uniformly on every set $A_\delta= \{z\in\D: |z -\xi| \ge \delta\}$, $\delta>0,$ and $g_{\xi, k} \to  0$ weakly as 
$k\to \infty$. 
We consider the two cases: $(i)$ when $\alpha =1$ and $(ii)$ when $0 < \alpha <1.$

\smallskip

$(i)$ It holds that $\sup_{k\in \mathbb Z_{\geq 1}}||g_{\xi,k}||_\infty<\infty.$  Since $u\in  M(\B(\D))$, we know that $u\in \B_0(\D)\cap H^\infty(\D).$   Let $B_\delta= \{z\in\D: |z -\xi| < \delta\},$ so $\D = A_\delta \cup B_\delta.$  Let $\varepsilon> 0$ be given, and  choose $\delta >0$ such that $|u(z)| < \varepsilon$ and $|u'(z)| (1 -|z|^2) < \varepsilon$ for $z\in B_\delta.$  The following estimates hold,

\begin{align*}
\norm{M_u(g_{\xi,k})}_\B   & \le   I_{k,{A_\delta}} + { II}_{k,{B_\delta}}+\abs{g_{\xi,k}(0)u(0)},
\end{align*}
where
\begin{align*}
 I_{k,{A_\delta}}&=\sup_{z\in A_\delta}|u(z)| |g'_{\xi,k}(z)| (1 - |z|^2)+\sup_{z\in A_\delta}|u'(z)| |g_{\xi,k}(z)| (1 - |z|^2)  \text{ and }\\
{ II}_{k,{B_\delta}} &=\sup_{z\in B_\delta}|u(z)| |g'_{\xi,k}(z)| (1 - |z|^2) +  \sup_{z\in B_\delta}|u'(z)| |g_{\xi,k}(z)| (1 - |z|^2).
\end{align*}
Consequently, we get that $\lim_{k\to\infty}  I_{k,{A_\delta}} =0$ and  $\lim_{k\to\infty} { II}_{k,{B_\delta}} \le  2\varepsilon.$ We also have $\abs{g_{\xi,k}(0)} \asymp 2^{-k}$.
Thus $\norm{{M_u}(g_{\xi,k})}_\B\to 0 \ \text{when} \ k\to\infty,$  which means by  Lemma 4.3.15 in \cite{D} that  $0\in\sigma_e(M_u)$. Therefore $M_u$ is not Fredholm. 

\smallskip

$(ii)$ The result follows similarily from showing that $\norm{u  g_{\xi, k}}_{\B_\alpha} \to  0$ as $k\to \infty$. Take $\varepsilon>0$ and choose $\delta>0$ such that $\abs{u(z)}<\varepsilon$ on $B_\delta$. It is clear that $I_{k,{A_\delta}}\to 0$ as $k\to \infty$. From the definition of $g_{\xi,k}$  we have $\norm{g_{\xi,k}}_\infty \asymp k^{\alpha-1} $, so $\norm{g_{\xi,k}}_\infty\to 0$ as $k\to \infty$, hence, ${ II}_{k,{B_\delta}}<2\varepsilon$ for $k$ large enough.

\end{proof}

Let us now consider the space $X(\D) = A_{\alpha,\beta}^p(\D)$  with $1 < p <\infty$. The following lemma will be used to obtain an estimate for the Bergman-Sobolev norm of the peak function.

\begin{lem}\label{someAsympBehav} 
Let $L,M\geq 0$. Then
\[
\frac{\Gamma(K+L)}{\Gamma(K)}  \sim K^L\text{ and } \frac{\Gamma(2K+L)}{\Gamma(K+L)\Gamma(K+M)}\sim \frac{2^{2K+L-1}}{\sqrt{\pi} }K^{\frac{1}{2}-M},
\]
as $K\to \infty$.
\end{lem}
\begin{proof}
According to Stirling's approximation, $\Gamma(x)\sim \sqrt{\frac{2\pi}{x}}\left(\frac{x}{e}\right)^x$ as $x\to\infty$, we have
\begin{align*}
\frac{\Gamma(K+L)}{K^L\Gamma(K)} &\sim  \sqrt{\frac{K}{K+L}} e^{K-(K+L)}\frac{(K+L)^{K+L}}{K^L K^K} \\
&=\left( 1+\frac{L}{K} \right)^{-\frac{1}{2}}  e^{-L} \left(1+\frac{L}{K}\right)^K\left(1+\frac{L}{K}\right)^L \rightarrow 1
\end{align*}
as $K\to\infty$. Moreover, 
\begin{align*}
\frac{\sqrt{\pi }K^{M-\frac{1}{2}}\Gamma(2K+L)}{2^{L+2K-1}\Gamma(K+L)\Gamma(K+M)} &\sim  \frac{   \sqrt{\pi}K^{M-\frac{1}{2}}     }{    2^{L+2K-1}     } \sqrt{\frac{(K+L)(K+M)}{2\pi(2K+L)}} \frac{e^M(2K+L)^{2K+L}}{(K+L)^{K+L} (K+M)^{K+M}} \\
&=e^M \sqrt{\frac{\left(1+\frac{L}{K}\right)\left(1+\frac{M}{K}\right)}{\left(1+\frac{L}{2K}\right)}}  \frac{\left(1+\frac{L}{2K}\right)^{2K+L}
}{\left(1+\frac{L}{K}\right)^{K+L}\left(1+\frac{M}{K}\right)^{K+M}}\to 1
\end{align*}
as $K\to\infty$.
\end{proof}

In the following important lemma a fairly good approximation of the behaviour of the Bergman-Sobolev norm of the peak functions is obtained. The proof also gives an exact asymptotic formula for $\norm{D^j f_{\xi,k}}_{A^p_\alpha}$ as $k\to \infty$ in the case of $p\in\mathbb Z_{\geq 1}$, namely, 
\[
\norm{D^j f_{\xi,k}}_{A^p_\alpha}^p\sim \frac{ \Gamma(\alpha+2) 2^{2\alpha+\frac{5}{2}-jp}}{   \sqrt{ \pi }   p^{\alpha+\frac{3}{2}}         }(k+1)^{jp-(\alpha+\frac{3}{2})}. 
\]
Furthermore, some properties for the normalized peak function are given in order to prove Lemma \ref{IfFredholm_BergmanSobolev}, from which a part of the main result follows. Observe that, as already mentioned in the beginning of section \ref{sect4}, the following lemma is a necessary refinement of Lemma 11  in \cite{CHZ} and, as a sharp estimate, it is also of independent interest. 
\begin{lem} \label{bra} 
Let $p \geq 1$, $\alpha > -1$ or $p = 2$, $\alpha =-1$. If $\beta\geq 0$, then
\[
\norm{f_{\xi,k}}_{A_{\alpha,\beta}^p}^p \asymp ( k+1)^{-\alpha + \beta p  -\frac{3}{2}}
\] 
for $k\in\mathbb Z$ large enough.
Consequently, if $\beta > \frac{2 +\alpha}{p}$ and $\xi\in \partial \mathbb D$, then the functions $g_{\xi, k} = f_{\xi,k} / \norm{f_{\xi,k}}_{A_{\alpha,\beta}^p} \in{\mathcal P}(\D)$  have the properties that $\norm{g_{\xi, k}}_{A_{\alpha,\beta}^p}=1$; $g_{\xi,k}\to 0$;  $R^mg_{\xi,k}\to 0$, $m\in\mathbb Z_{\geq 1}$,  uniformly on every set $A_\delta= \{z\in\D: |z -\xi| \ge \delta\}$, $\delta>0,$ and for $p>1$ it also holds that $g_{\xi, k} \to  0$ weakly in $A_{\alpha,\beta}^p$ as $k\to \infty$.
\end{lem}

\begin{proof} 
Let  $N$ be a positive integer satisfying $N > \beta- \frac{\frac{1}{2} +\alpha}{p}$. By (\ref{formimp}) and   (\ref{good}) we have that 
\begin{equation*}
\norm{f_{\xi,k}}_{A_{\alpha,\beta}^p}\asymp\norm{f_{\xi,k}}_{A_{(N-\beta)p+\alpha,N}^p}\asymp \sum_{l=0}^{N-1} |D^l f_{\xi,k}(0)| + \norm{D^Nf_{\xi,k}}_{A_{(N-\beta)p+\alpha}^p} \asymp \norm{D^Nf_{\xi,k}}_{A_{(N-\beta)p+\alpha}^p}.
\end{equation*}
The last equivalence follows from
\[
0\leq \abs{ D^l f_{\xi,k}(0) } \leq \abs{ D^N f_{\xi,k}(0) } \leq  \norm{D^N f_{\xi,k}}_{A_{(N-\beta)p+\alpha}^p}
\] 
for $l\leq N$. To finish the proof, it will be shown that for $\gamma>-1$ and $j\in\mathbb Z_{\geq 0}$ we have
\[
 \norm{D^j f_{\xi,k}}_{A_\gamma^ p}^p \asymp (k+1)^{jp-(\gamma+\frac{3}{2})} ,
\]
from which the lemma follows by letting $\gamma = (N-\beta)p + \alpha$ and $j=N$.

Let $q$ be the smallest integer greater than or equal to $p$ and $k\geq j$. We have
{
\allowdisplaybreaks
\begin{align*}
\frac{\norm{D^j f_{\xi,k}}_{A_\gamma^ p}^p}{\gamma+1}&=\int_{\mathbb D} |D^j f_{\xi,k} (z)|^p (1-|z|^2)^\gamma  dA(z) \\
&= \left(\frac{k!}{(k-j)!}\right)^p\int_{\mathbb D} \left(\frac{    \left|1+\conj{\xi} z\right|^{(k-j)}    }{2^{k}}\right)^p (1-|z|^2)^\gamma  dA(z) \\
&\stackrel{(*)}{\geq} \left(\frac{k!}{(k-j)!}\right)^p\int_{\mathbb D} \left(\frac{    \left|1+\conj{\xi} z\right|^{(k-j)}    }{2^{k}}\right)^q (1-|z|^2)^\gamma  dA(z)  \\
&\stackrel{(**)}{\geq} \left(\frac{k!}{(k-j)!}\right)^p\int_{\mathbb D} \frac{1}{     2^{jq}  }    \frac{     \left|1+\conj{\xi} z\right|^{2K}          }{    2^{2K}           } (1-|z|^2)^\gamma  dA(z)=U_{k,\xi,j,\gamma}. 
\end{align*}
}The $(*)$ indicates that choosing $q$ to be the greatest integer smaller than $p$ we similarily obtain the opposite strict inequality. The function $K\colon\mathbb Z_{\geq j}\to \mathbb Z_{\geq 0}$ is defined as $K=K_j(k)=\frac{(k-j)q}{2}$ if $k-j$ is even. In this case $\stackrel{(**)}{\geq}$ is an equality. If $k-j$ is odd, then $K$ is defined by $K=\frac{(k+1-j)q}{2}$ or $K=\frac{(k-1-j)q}{2}$ depending on which inequality we want to obtain. In the latter case $\stackrel{(**)}{\geq}$ is replaced by $\leq$. 

We continue the proof by evaluating the integral with respect to the angle. It is enough to examine the expression for $\xi = 1$. Now consider the functions $g_r\in L^2([0,2\pi)), \ g_r(t)=(1+re^{it})^K=\sum_{n=0}^K\binom{K}{n} r^n e^{itn}$ for $r\geq 0$. From Parseval's equality we obtain
\begin{align*}
\int_0^{2\pi} |1+re^{it}|^{2K} dt = 2\pi \sum_{n=0}^K \binom{K}{n}^2 r^{2n},
\end{align*}
for every $0\leq r<1$ and thus,
{
\allowdisplaybreaks
\begin{align*}
U_{k,\xi,j,\gamma} &=\left(\frac{k!}{(k-j)!}\right)^p \frac{2}{2^{(jq+2K)}} \int_0^1    \sum_{n=0}^K \binom{K}{n}^2 r^{2n} (1-r^2)^\gamma   rdr\\
&=\left(\frac{k!}{(k-j)!}\right)^p \frac{1}{2^{(jq+2K)}}   \sum_{n=0}^K \binom{K}{n}^2 \int_0^1 r^{2n} (1-r^2)^\gamma   2rdr\\
&=\left(\frac{k!}{(k-j)!}\right)^p \frac{1}{2^{(jq+2K)}}   \sum_{n=0}^K \binom{K}{n}^2 \int_0^1 r^n (1-r)^\gamma   dr.
\end{align*}
}Moreover,
{
\allowdisplaybreaks
\begin{align*}
U_{k,\xi,j,\gamma} &= \left(\frac{k!}{(k-j)!}\right)^p\frac{1}{2^{(jq+2K)}}   \sum_{n=0}^K \binom{K}{n}^2 \beta(n+1,\gamma+1)\\
&=\left(\frac{k!}{(k-j)!}\right)^p \frac{1}{2^{(jq+2K)}}   \sum_{n=0}^K \binom{K}{n}^2\frac{\Gamma(\gamma+1)\Gamma(n+1)}{\Gamma(n+\gamma+2)}\\
&=\left(\frac{\Gamma(k-j+1+j)}{\Gamma(k-j+1)}\right)^p  \frac{1}{2^{(jq+2K)}}  \frac{\Gamma(\gamma+1)}{\Gamma(K+\gamma+2)}    \frac{\Gamma(2K+\gamma+2)}{\Gamma(K+\gamma+2)}\\
&\sim k^{jp} \frac{ \Gamma(\gamma+1)}{2^{(jq+2K)}}  K^{-\gamma-\frac{3}{2}} \frac{2^{\gamma+1+2K}}{\sqrt{\pi }}\\
&\sim \Gamma(\gamma+1) k^{jp} (kq)^{-\gamma-\frac{3}{2}} \frac{2^{2\gamma+\frac{5}{2}-jq}}{\sqrt{\pi }} \\
&=\frac{ \Gamma(\gamma+1) 2^{2\gamma+\frac{5}{2}-jq}}{   \sqrt{ \pi }   q^{\gamma+\frac{3}{2}}         }k^{jp-(\gamma+\frac{3}{2})},
\end{align*}
}as $k\to\infty$, where the first asymptotic approximation is given by Lemma \ref{someAsympBehav} and $(k-c)^a\sim k^a$ as $k\to\infty$ for every $c\in\mathbb R$. The third equality follows from the Chu-Vandermonde identity, see \cite[p.~32]{K} with the parameters $n=K,b=-K$ and $c=\gamma+2$.

To prove that $(g_{\xi,k})_{k=1}^\infty$ is a weak null sequence, let $B_{A^p_{\alpha,\beta}(\mathbb D)}$ denote  the closed unit ball  of the Bergman-Sobolev space $A^p_{\alpha,\beta}(\mathbb D),\ p>1$. Let $\tau_p$ denote the topology of pointwise convergence. Notice that $(B_{A^p_{\alpha,\beta}(\mathbb D)},\tau_p)$ is a Hausdorff space and that $B_{A^p_{\alpha,\beta}(\mathbb D)}$ is weakly compact, since the space is reflexive. Since $\delta_z\in A^p_{\alpha,\beta}(\mathbb D)^*$ by condition \ref{finerThanCOTop}, the identity map 
\[
\id\colon(B_{A^p_{\alpha,\beta}(\mathbb D)},w)\to (B_{A^p_{\alpha,\beta}(\mathbb D)},\tau_p)
\]
is continuous, and hence,  it represents a homeomorphism between the spaces $(B_{A^p_{\alpha,\beta}(\mathbb D)},w)$ and $(B_{A^p_{\alpha,\beta}(\mathbb D)},\tau_p)$. Since $\id^{-1}\colon (B_{A^p_{\alpha,\beta}(\mathbb D)},\tau_p) \to (B_{A^p_{\alpha,\beta}(\mathbb D)},w)$ is continuous, we conclude that $g_{\xi, k}\to 0$ weakly, when $k\to\infty$.

\end{proof}

\begin{lem}\label{IfFredholm_BergmanSobolev}
Let $\alpha>-1$, $p>1$ or $\alpha=-1$, $p=2$ and assume $\beta > \frac{2+\alpha}{p}$. 
If $M_u\colon A^p_{\alpha,\beta} \to A^p_{\alpha,\beta}$ is Fredholm, then there exist $\delta >0$ and $r \in (0,1)$ such that $|u(z)| \ge \delta$ for all $r \le |z| < 1$. 
\end{lem}
\begin{proof}
The proof will be carried out by contraposition. Since $u$ belongs to the disk algebra  it is  continuous  up to the boundary of $\mathbb D.$ Assume there is a point $\xi\in \partial \mathbb D$ such that $u(\xi)=0$. This assumption is equivalent to $u$ not being bounded from below arbitrarily close to the boundary, since $u$ is continuous. It will be shown that 
\[
\norm{ug_{\xi,k}}_{A^p_{p(N-\beta)+\alpha,N}}\to 0  \ \text{ as} \ k\rightarrow \infty,
\] which by (\ref{formimp}) implies that
\begin{equation}\label{ugNullSeq}
\norm{ug_{\xi,k}}_{A^p_{\alpha,\beta}}\to 0  \ \text{ as} \ k\rightarrow \infty,
\end{equation}
 where $N$ is the positive integer satisfying
\[
0<N-\beta + \frac{\frac{1}{2}+\alpha}{p}\leq 1,
\]
and $g_{\xi,k}$ is the function defined in Lemma \ref{bra}. The lemma follows from Lemma \ref{bra}, (\ref{ugNullSeq}) and Lemma 4.3.15 in \cite{D}.

To prove the null sequence statement, we will make use of (\ref{BergSobNormEquivRel}). First, notice that by Lemma \ref{bra} we obtain
\begin{align*}
\abs{u(0)g_{\xi,k}(0)}\lesssim \frac{   \abs{u(0)f_{\xi,k}(0)}   }{(k+1)^{-\frac{\alpha}{p}+\beta-\frac{3}{2p}}   } \to 0
\end{align*}
as $k\to \infty$. Using the general Leibniz formula we have
\[
R^N(u g_{\xi,k}) = \sum_{j=0}^N  \binom{N}{j} R^j u R^{N-j}  g_{\xi,k},
\]
from which it follows that 
\begin{equation}
\label{eq: leibniz}
\norm{R^N(u g_{\xi,k}) }_{A^p_{p(N-\beta)+\alpha}} \le \sum_{j=0}^N  \binom{N}{j}\norm{ R^j u R^{N-j}  g_{\xi,k} }_{A^p_{p(N-\beta)+\alpha}}.
\end{equation}
Therefore, it suffices to show that 
\[
I_{k,j} = \int_\mathbb D |R^j u R^{N-j} g_{\xi,k}|^p dA_{p(N-\beta)+\alpha}
\]
approaches zero for $j=0,1,...,N$ as $k$ tends to infinity. To prove the assertion for the case $j=0$, we take $\varepsilon>0$ and choose $\delta>0$ such that $|u(z)|^p<\varepsilon$ for all 
\[
z\in B_\delta= \{z\in\D: |z -\xi| < \delta\}.
\] 
We can now choose a $K>0$ such that
\[
\int_{A_\delta} |R^N g_{\xi,k}(z) | ^p dA_{p(N-\beta)+\alpha} (z)  < \varepsilon,
\]
which implies
\[
\int_{A_\delta} |u(z)R^N g_{\xi,k}(z) | ^p dA_{p(N-\beta)+\alpha} (z)  < \norm{u}_\infty^p \varepsilon
\]
for $k>K$, where Lemma \ref{bra} has been used and $A_\delta= \{z\in\D: |z -\xi| \geq \delta\}$. Thus, for $k > K$
\[
I_{k,0} <( \norm{u}^p_{\infty} + \norm{R^N g_{\xi,k}}_{A^p_{p(N-\beta)+\alpha} }^p)\varepsilon \leq ( \norm{u}^p_{\infty} + M\norm{g_{\xi,k}}_{A^p_{\alpha,\beta} }^p)\varepsilon
\]
where (\ref{formimp}) gives the second inequality for some $M>0$. Since $u\in A_{\alpha,\beta}^p(\mathbb D)\subset H^\infty(\mathbb D)$ and $\|g_{\xi,k}\|_{A^p_{\alpha, \beta}}=1$ 
for every $k$, the result follows.
To assure the result in the case $j\geq 1$, we will use the following approximation:

\begin{align*}
I_{k,j}&\leq \frac{(k+1)^{p(N-j)}}{\norm{f_{\xi,k}}^p_{A_{\alpha,\beta}^p}} \int_\mathbb D \abs{R^j u(z)}^p \left|\frac{1+\conj{\xi}z}{2}\right|^{(k-(N-j))p}  dA_{p(N-\beta)+\alpha}(z).
\end{align*}
From Lemma \ref{bra} it follows that
\begin{align}
\begin{split}
\label{IntegralIApproxMAIN}
I_{k,j}&\lesssim \frac{    (k+1)^{p(N-j)}    }{      ( k+1)^{   -\alpha + \beta p  -\frac{3}{2} }       } \int_\mathbb D |R^j u(z)|^p \left|\frac{1+\conj{\xi}z}{2}\right|^{(k-(N-j))p} (1-|z|^2)^{p(N-\beta)+\alpha} dA(z)\\
&=  (k+1)^{p(N-\beta+\frac{\alpha +\frac{3}{2}   }{ p}-j)}       \int_\mathbb D |R^j u(z)|^p \left|\frac{1+\conj{\xi}z}{2}\right|^{(k-(N-j))p} (1-|z|^2)^{p(N-\beta)+\alpha} dA(z).
\end{split}
\end{align}

For integers $j\in[2,N]$ the result $I_{k,j}\to 0$ as $k\to\infty$ is obtained from the following three facts:

\begin{align*}
&u\in A^p_{\alpha,\beta} \simeq A^p_{p(N-\beta)+\alpha,N}\subset A^p_{p(N-\beta)+\alpha,j}; \\
&\norm{f_{\xi,k} }_\infty \leq 1 \ \forall k\in\mathbb Z_{\geq 1}; \\
&p(N-\beta+\frac{ \frac{3}{2} + \alpha}{p}-j)\leq 1+p-pj<0.
\end{align*}

For $j=1$ we make an additional partition. We will, once at a time, assume that $N-\beta+\frac{\frac{3}{2} +  \alpha }{p}-1$ is stricly less than zero, equal to zero or strictly larger than zero. In the first case we can apply the procedure used for $j\geq 2$. In the second case we may utilize the Lebesgue dominated convergence theorem to functions 
\[
\left|\frac{1+\conj{\xi}z}{2}\right|^{(k-(N-j))p} \le 1
\] 
for all $z \in \D$ and $k \in \mathbb Z_{\geq N}$ to obtain the result. 

The only thing that remains to show is that $I_{k,1}\to 0$ as $k\to \infty$ when $N-\beta+\frac{\frac{3}{2} + \alpha }{p}-1>0$. This condition implies that  
\[
N >\beta - \frac{\frac{3}{2} +  \alpha  }{p}+1 >1, 
\]
 so that $N\geq 2$. 

To prove that $(I_{k,1})_{k=1}^\infty$ is a null sequence we will use Lemma 5.4 in \cite{BB} and Lemma \ref{bra}. Lemma 5.4 in \cite{BB} gives us three different approximations for the behaviour of $|Du(z)|$, depending on values of some parameters. Hence, it suffices to prove the null convergence for all of these approximations, one at a time. Notice that $q=\alpha+1$ when comparing notations with \cite{BB}. First, assume $\beta<\frac{2+\alpha}{p}+ 1$. Then we have
{
\allowdisplaybreaks
\begin{align*}
I_{k,1} &\lesssim \norm{u}_{A^p_{\alpha,\beta}}^p\int_{\D} |R^{N-1} g_{\xi,k}|^p (1-|z|^2)^{-p(\frac{2+\alpha}{p}+1-\beta)}dA_{p(N-\beta)+\alpha}(z)\\
&\lesssim\frac{\norm{u}_{A^p_{\alpha,\beta}}^p}{\norm{f_{\xi,k}}_{A^p_{\alpha,\beta}}^p}\int_{\D} |R^{N-1} f_{\xi,k}|^p dA_{p(N-1)-2}(z)\\
&=\frac{\norm{u}_{A^p_{\alpha,\beta}}^p}{\norm{f_{\xi,k}}_{A^p_{\alpha,\beta}}^p}\norm{f_{\xi,k}}^p_ {A^p_{p(N-1)-2,N-1}}\\
&\asymp\norm{u}_{A^p_{\alpha,\beta}}^p(k+1)^{2  +\alpha-\beta p   },
\end{align*}
}therefore $(I_{k,1})_k$ is a null sequence in this case. 
If $\beta\geq \frac{2+\alpha}{p}+ 1$, then a worse upper bound than the one stated in Lemma 5.4 is given by $C   \norm{u}_{A^p_{\alpha,\beta}}^p \frac{        1       }{(1-|z|^2)^r}$ for some positive constant $C$ and any $r>0$. In this case we have, for $0<r<\frac{1}{2}$, that
\begin{align*}
I_{k,1} &\lesssim \norm{u}_{A^p_{\alpha,\beta}}^p\int_{\D} |R^{N-1} g_{\xi,k}|^p (1-|z|^2)^{-r}dA_{p(N-\beta)+\alpha}(z)\\
&\lesssim\frac{\norm{u}_{A^p_{\alpha,\beta}}^p}{\norm{f_{\xi,k}}_{A^p_{\alpha,\beta}}^p}\int_{\D} |R^{N-1} f_{\xi,k}|^p dA_{p(N-\beta)+\alpha-r}(z)\\
&=\frac{\norm{u}_{A^p_{\alpha,\beta}}^p}{\norm{f_{\xi,k}}_{A^p_{\alpha,\beta}}^p}\norm{f_{\xi,k}}^p_ {A^p_{p(N-\beta)+\alpha-r,N-1}}\\
&\asymp\norm{u}_{A^p_{\alpha,\beta}}^p(k+1)^{r-p},
\end{align*}
which completes the proof.

\end{proof}

We are now ready to present the main result.

\begin{thm}\label{MainResEssSpec}  Let  $X(\mathbb D)$  be any of the following spaces:
\begin{enumerate}[label=\upshape(\alph*)]
\itemsep-0.6em 
\item $\B_\alpha(\mathbb D), \, 0 <\alpha  <  1,$ with $u\in M(\B_\alpha(\D)) =\B_\alpha(\D) \subset  A(\mathbb D)$;\\
\item $\B(\mathbb D)$  with  $u\in M(\B(\D))\cap  A(\mathbb D)$;\\
\item $A^p_{\alpha,\beta}(\mathbb D)$ with  $u\in M(A^p_{\alpha,\beta}(\mathbb D)) =A^p_{\alpha,\beta}(\mathbb D)\subset  A(\mathbb D)$, where $p>1, \, \alpha>-1\text{ and } \beta > \frac{2+\alpha}{p}$;\\
\item $H^2_\beta(\mathbb D)$ with  $u\in M(H^2_\beta(\mathbb D)) =H^2_\beta(\mathbb D)\subset  A(\mathbb D),$ where $\beta > \frac{1}{2}$.
\end{enumerate}
\medskip
\noindent Then the essential spectrum of  $M_u:X(\mathbb D) \to X(\mathbb D)$  is given by 
\[
\sigma_e(M_u) = \bigcap_{0 < r < 1} \closed{u(\mathbb D \setminus r\mathbb D)}=u(\partial \mathbb D).
\]
\end{thm}
\begin{proof}
As in the proof of Theorem \ref{EssSpecMlutiHInfty},  now using lemmas \ref{ThenFredholm_ForX}, \ref{IfFredholm_BlochType}  and \ref{IfFredholm_BergmanSobolev},  we obtain $\sigma_e(M_u) = \bigcap_{0 < r < 1} \closed{u(\mathbb D \setminus r\mathbb D)}$ whenever $u\in M(X(\mathbb D))\cap A(\mathbb D)$ and $X(\mathbb D)$ is any of the spaces listed above. 
To prove the last equality, we utilize the continuity of $u$ on $\overline{\D}$, which implies the first equality below 
\begin{align*}
\bigcap_{0 < r < 1} \closed{u(\mathbb D \setminus r\mathbb D)}&=\bigcap_{0 < r < 1} u(\closed{\mathbb D \setminus r\mathbb D})\\
&=\bigcap_{0 < r < 1} u(\closed{\mathbb D} \setminus r\mathbb D)\\
&\supset u(\partial \mathbb D).
\end{align*}
To show the opposite inclusion, take $z\in\bigcap_{0 < r < 1} u(\closed{\mathbb D} \setminus r\mathbb D)$. Now there is a sequence $(y_n)_{n=1}^\infty, \ 1-\frac{1}{n}\leq|y_n|\leq 1$ such that $u(y_n) = z$. Since $(y_n)_{n=1}^\infty$ is bounded there is a convergent subsequence $(y_{n_k})_{k=1}^\infty$ such that $y_{n_k}\to y\in \partial \mathbb D$ as $k\to \infty$. Since $u$ is continuous on
$\closed{\mathbb D}$ we have
\[
z=\lim_{k\to\infty} u(y_{n_k}) = u(y), 
\]
so $z\in u(\partial \mathbb D)$, which proves the theorem.
\end{proof}

\bigskip

\subsection*{Acknowledgements}

The first two authors were supported in part by the Academy of Finland project 296718. The third author acknowledges support from the Magnus Ehrnrooth Foundation.

\end{document}